\documentclass[10pt]{article}
\usepackage{amsmath, amssymb, amsfonts, amsthm, epsfig, float, graphicx, stix}
\usepackage[all]{xy}
\title{Topological full groups of minimal subshifts 
and quantifying local embeddings into finite groups}
\author{Henry Bradford and Daniele Dona}

\newtheorem{thm}{Theorem}[section]
\newtheorem{lem}[thm]{Lemma}
\newtheorem{propn}[thm]{Proposition}
\newtheorem{coroll}[thm]{Corollary}
\newtheorem{defn}[thm]{Definition}
\newtheorem{ex}[thm]{Example}
\newtheorem{notn}[thm]{Notation}

\newtheorem{rmrk}[thm]{Remark}

\DeclareMathOperator{\Alt}{Alt}

\DeclareMathOperator{\Cyl}{Cyl}

\DeclareMathOperator{\DCyl}{DCyl}
\DeclareMathOperator{\GL}{GL}
\DeclareMathOperator{\Homeo}{Homeo}

\DeclareMathOperator{\im}{im}

\DeclareMathOperator{\Schr}{Schr}

\DeclareMathOperator{\supp}{supp}

\DeclareMathOperator{\Sym}{Sym}

\begin{document}

\maketitle

\begin{abstract}
We investigate quantitative aspects of the LEF property 
for subgroups of the topological full group $\lBrack \sigma \rBrack$ 
of a two-sided minimal subshift over a finite alphabet, 
measured via the LEF growth function. 
We show that the LEF growth of $\lBrack \sigma \rBrack^{\prime}$ 
may be bounded from above and below in terms of the 
recurrence function and the complexity function of the subshift, respectively. 
As an application, we construct groups of previously unseen 
LEF growth types, 
and exhibit a continuum of finitely generated LEF groups 
which may be distinguished from one another by their LEF growth. 
\end{abstract}

\section{Introduction}

Often in geometric group theory, 
one considers a \emph{growth function} $F_{\Gamma}$, 
which describes some part of the asymptotic structure 
of a finitely generated group $\Gamma$. 
Examples include \emph{subgroup growth}; 
\emph{word growth}; \emph{conjugacy growth}; \emph{Dehn function}; 
\emph{F\o lner function}; \emph{residual finiteness growth}. 
Having introduced $F_{\Gamma}$, 
it is always natural to attempt the following: 
(i) estimate $F_{\Gamma}$ for some group $\Gamma$ of interest; 
(ii) relate the behaviour of $F_{\Gamma}$ 
to structural features of $\Gamma$ or its actions; 
(iii) explore the types of functions which can arise 
as $F_{\Gamma}$ for some $\Gamma$. 
In this paper we make contributions to all three of these themes 
for the \emph{LEF growth} function, 
by examining some groups arising in Cantor dynamics. 

\subsection{Statement of results}

A group $\Gamma$ is \emph{LEF} 
(locally embeddable into finite groups) if every finite subset of $\Gamma$ 
admits an injective partial homomorphism (a local embedding) 
into a finite group. 
In other words, every finite subset of the multiplication table 
of $\Gamma$ occurs in the multiplication table of some finite group. 
If $\Gamma$ is finitely generated by $S$, 
the prototypical finite subsets are the balls $B_S(n)$ 
in the associated word metric. The \emph{LEF growth function} 
$\mathcal{L}_{\Gamma} ^S$ sends $n\in\mathbb{N}$
 to the minimal order of a finite group 
into which $B_S(n)$ locally embeds. 
The dependence of the function on $S$ is slight, 
so we suppress $S$ from our notation for the rest of the Introduction. 

That the topological full group $\lBrack \varphi \rBrack$ 
of a Cantor minimal system $(X,\varphi)$ is a LEF group 
was proved by Grigorchuk and Medynets \cite{GrigMedy}. 
We prove an effective version of their result, 
in the case of a two-sided minimal subshift $(X,\sigma)$ 
over a finite alphabet. 
Let $R_X : \mathbb{N} \rightarrow \mathbb{N}$ be 
the \emph{recurrence function} of $(X,\sigma)$ 
(see Definition \ref{RecurrenceDefn} below). 

\begin{thm} \label{UBmainthm}
Let $\Gamma$ be a finitely generated subgroup of $\lBrack \sigma \rBrack$. 
Then: 
\begin{equation}
\mathcal{L}_{\Gamma} (n) \preceq \big( 2 R_X(n) \big)!\text{.}
\end{equation}
In particular this inequality holds for 
$\Gamma = \lBrack \sigma \rBrack^{\prime}$. 
\end{thm}

For nondecreasing unbounded functions 
$F_1$ and $F_2$ we write $F_1 \preceq F_2$ if, 
up to constant rescaling of the argument, $F_2$ bounds $F_1$ from above 
(see Definition \ref{FunctionCompDefn} below). 
Note that the Cantor minimal system 
$(X,\varphi)$ being isomorphic to a minimal subshift 
is a necessary and sufficient condition for $\lBrack \varphi \rBrack^{\prime}$ 
to be finitely generated \cite{Mat}. 
In the other direction, we have the following lower bound 
on the LEF growth of $\lBrack \sigma \rBrack^{\prime}$. 
Let $p_X : \mathbb{N} \rightarrow \mathbb{N}$ be the 
\emph{complexity function} of the minimal subshift $(X,\sigma)$ 
(see Definition \ref{ComplexityDefn} below). 

\begin{thm} \label{LBmainthm}
There exists $c>0$ such that $\exp \big(c p_X (n^{1/2} ) \big) 
\preceq \mathcal{L}_{\lBrack \sigma \rBrack^{\prime}} (n)$. 
\end{thm}

Our proof yields $c=\log(60)/9 \approx 0.455$, 
but small modifications to the argument would enable us to 
make $c$ arbitrarily large. 
The group $\lBrack \sigma \rBrack^{\prime}$ always has exponential 
word growth, which immediately implies that 
$\mathcal{L}_{\lBrack \sigma \rBrack^{\prime}} $ 
grows at least exponentially. 
One consequence of Theorem \ref{UBmainthm} is that 
when the subshift $X$ is extremely ``orderly'', 
then this exponential lower bound is close to best-possible. 

\begin{ex}
Let $X$ be a linearly recurrent subshift. Then: 
\begin{equation}
\exp (n) \preceq \mathcal{L}_{\lBrack \sigma \rBrack^{\prime}}(n) \leq n!\text{.}
\end{equation}
\end{ex}

On the other hand, is $X$ is highly non-deterministic, 
then Theorem \ref{LBmainthm} gives a novel lower bound for 
the LEF growth of $\lBrack \sigma \rBrack^{\prime}$. 

\begin{ex}
Let $X$ be a subshift of positive entropy. Then: 
\begin{equation}
\exp(\exp(n^{1/2})) \preceq 
\mathcal{L}_{\lBrack \sigma \rBrack^{\prime}} (n)\text{.}
\end{equation}
\end{ex}

See Subsection \ref{WordsSubsect} for definitions 
of linear recurrence and entropy. 
Although the upper and lower bounds proved in Theorems \ref{UBmainthm} 
and \ref{LBmainthm} are some distance apart, 
they are powerful enough to allow us to observe 
new phenomena in the kinds of functions which can 
arise as the LEF growth functions of groups. 

\begin{thm} \label{IntGrowthmainthm}
For any $r \in [ 2,\infty )$, 
there exists a finitely generated LEF group $\Gamma^{(r)}$ 
such that: 
\begin{itemize}
\item[(i)] There exists $C_r > 0$ such that $\mathcal{L}_{\Gamma^{(r)}} (n) 
\preceq \exp \Big( \exp \big(C_r (\log n)^r \big) \Big)$; 
\item[(ii)] For any $2 \leq r^{\prime} < r$, and for all $C>0$, 
$\mathcal{L}_{\Gamma^{(r)}} (n) 
\npreceq \exp \Big( \exp \big( C(\log n)^{r^{\prime}} \big) \Big)$. 
\end{itemize} 
\end{thm}

Previously, the only functions known to arise as LEF growth functions 
of groups were polynomial, exponential, or comparable to 
$\exp(n^C)$ or $\exp(\exp(n))$ (see \cite[Theorem 1.7]{Brad}) 
and some inexplicit, very large functions 
(see \cite[\S 5]{Brad}). 
Theorem \ref{IntGrowthmainthm} also immediately implies that 
uncountably many inequivalent LEF growth functions occur 
among finitely generated LEF groups, 
answering a question posed in \cite{Brad}. 

\begin{thm} \label{UnctbleThm}
There is an uncountable family $\mathcal{F}$ of 
pairwise nonisomorphic finitely generated LEF groups such that, 
for $\Gamma_1 , \Gamma_2 \in \mathcal{F}$, 
if $\Gamma_1 \neq \Gamma_2$ then $\mathcal{L}_{\Gamma_1} \napprox \mathcal{L}_{\Gamma_2}$.
\end{thm}

We write $F_1 \approx F_2$ if $F_1 \preceq F_2$ 
and $F_2 \preceq F_1$ (see Definition \ref{FunctionCompDefn} below). 
The groups constructed in Theorem \ref{IntGrowthmainthm} 
have the form $\lBrack \sigma \rBrack^{\prime}$, 
for $(X_r,\sigma)$ a minimal subshift which we construct by adapting 
a construction of Jung, Lee and Park \cite{JLP16}, as follows. 

\begin{thm} \label{IntGrowthShiftThm}
For every $r \in [ 2,\infty )$, there are a minimal subshift 
$(X_r , \sigma)$, constants $C_r , c_r >0$, 
and an increasing sequence $(n_i ^{(r)})$ of integers such that: 
\begin{itemize}
\item[(i)] For all $n\geq 2$, 
$R_{X_r} (n) \leq \exp \big( C_r (\log n)^r \big)$; 
\item[(ii)] For all $i \in \mathbb{N}$, 
$p_{X_r} (n_i ^{(r)}) \geq \exp \big( c_r (\log n_i ^{(r)} )^r \big)$. 
\end{itemize}
\end{thm}

The conclusions (i) and (ii) of Theorem \ref{IntGrowthmainthm} 
for the group $\Gamma^{(r)} = \lBrack \sigma \rBrack^{\prime}$ then follow 
from Theorem \ref{IntGrowthShiftThm} (i) and (ii), 
by Theorems \ref{UBmainthm} and \ref{LBmainthm}, respectively. 
Further examples of subshifts of ``intermediate'' complexity 
could be a rich source of examples of new exotic behaviours 
in the LEF growth of groups, and this should be investigated further. 

\subsection{Background and structure of the paper}

The concept of a LEF group first appears in the work of Mal'cev, 
but was developed and popularised by Vershik and Gordon \cite{VersGord}. 
All residually finite groups are LEF, including all finitely generated 
nilpotent or linear groups, but LEF enjoys some closure properties 
that residual finiteness does not: for instance, 
the (regular restricted) wreath product of LEF groups is LEF. 
Among finitely presented groups, 
the classes of LEF and residually finite groups coincide, 
and this observation provides a useful tool for proving that 
certain groups are not finitely presentable 
(see \cite{GrigMedy} for a proof along these lines for 
derived subgroups of topological full groups). 
LEF groups have also been studied in connection with weaker approximation 
properties of groups, such as soficity and hyperlinearity, 
since they provide a source of examples beyond 
those arising from residual finiteness or amenability. 
For instance, Elek and Szab\'o \cite{ElekSzab} used the LEF property 
to construct the first examples of sofic groups 
which are not residually amenable. 

The LEF growth function was introduced independently in 
\cite{ArzChe} and \cite{BouRabStu} (in the latter under the name 
\emph{geometric full residual finiteness growth}), 
and fits into the extensive literature on quantifying finite approximations 
of infinite groups which has developed over the last decade. 
This program started with the work of Bou-Rabee 
and collaborators on quantitative residual finiteness
(see \cite{BoRaCheTim} and the references therein). 
Using results on quantitative residual finiteness, 
word growth, and finite presentability, 
the LEF growth function has been estimated 
for several natural classes of groups 
(see \cite[\S 2.4]{Brad}). 

\begin{ex} \label{SmallLEFGrowthEx}
Let $\Gamma$ be a finitely generated group. 
\begin{itemize}
\item[(i)] If $\Gamma$ is virtually $\mathbb{Z}^d$, 
then $\mathcal{L}_{\Gamma} (n) \approx n^d$; 
\item[(ii)] $\mathcal{L}_{\Gamma}$ is bounded above by a polynomial function 
iff $\Gamma$ is virtually nilpotent; 
\item[(iii)] If $\Gamma \leq \GL_d (\mathbb{Z})$ 
is finitely generated, not virtually nilpotent, 
then $\mathcal{L}_{\Gamma} (n) \approx \exp(n)$. 
\end{itemize}
\end{ex}

Groups of larger LEF growth can be explicitly constructed using wreath products. 

\begin{thm}[\cite{Brad} Theorem 1.8] \label{WPLEFGrthThm}
If $\Gamma$ is a finitely generated LEF group with word growth function 
$\gamma_{\Gamma}$, 
and $\Delta$ is a finite centreless group, 
then $\exp\big( \gamma_{\Gamma} (n) \big) \preceq \mathcal{L}_{\Delta \wr \Gamma} (n) \preceq \exp\big( \mathcal{L}_{\Gamma} (n) \big)$. 
\end{thm}

In particular, using Example \ref{SmallLEFGrowthEx}, 
Theorem \ref{WPLEFGrthThm} allows us to construct 
groups of LEF growth $\approx \exp (\exp(n))$ 
and $\approx \exp(n^d)$ (for any $d \in \mathbb{N}$). 
Beyond this, however, few types of LEF growth functions 
had hitherto been observed, and our Theorem \ref{IntGrowthmainthm} 
greatly extends the spectrum of known growth types.  

The derived subgroup of the topological full group 
$\lBrack\sigma\rBrack$ of a minimal subshift $(X,\sigma)$ 
is a remarkable object in group theory. 
It is a finitely generated infinite simple group, 
which, as well as being LEF, is amenable \cite{JusMon}
(and indeed was the first group discovered with this combination of properties). 
It is also a natural invariant from the point of view of 
topological dynamics. 
As shown in \cite{BezMed}, 
for any Cantor minimal system $(X,\varphi)$, 
$\lBrack\varphi\rBrack^{\prime}$ 
retains perfect information about the 
dynamics of $(X,\varphi)$. 

\begin{thm}[Bezuglyi-Medynets]
Let $(X,\varphi)$ and $(Y,\psi)$ be Cantor minimal systems. 
Then $\lBrack\varphi\rBrack^{\prime} \cong\lBrack\psi\rBrack^{\prime}$ 
iff $(X,\varphi)$ and $(Y,\psi)$ are flip-conjugate. 
\end{thm}

It is therefore reasonable to expect that 
group-theoretic asymptotic invariants of $\lBrack\varphi\rBrack^{\prime}$ 
should reflect asymptotic features of the dynamical system $(X,\varphi)$. 
Our Theorems \ref{UBmainthm} and \ref{LBmainthm} are in this spirit: 
knowing the LEF growth function of $\lBrack\sigma\rBrack^{\prime}$ 
allows one to deduce some bounds on the recurrence or complexity 
functions of $X$. 

Our proof of Theorem \ref{UBmainthm} is based on Elek's streamlined 
proof of LEF for topological full groups \cite{Elek}. 
Given $\Gamma \leq \lBrack \sigma \rBrack$, 
finitely generated by $S$, and a two-sided sequence $\mathbf{x} \in X$, 
the $\sigma$-orbit $\mathcal{O}$ of $\mathbf{x}$ is dense in $X$, 
so $\Gamma$ acts faithfully on $\mathcal{O}$. 
Further, any short word in $S$ moves some cylinder set $C \subseteq X$, 
defined by a short string in $\mathbf{x}$. 
Since only finitely many such cylinder sets $C$ arise, 
each of which intersects $\mathcal{O}$, 
there exists $M \in \mathbb{N}$ such that no nonidentity element of $B_S (n)$ 
fixes $\lbrace \sigma^i \mathbf{x} : 1 \leq i \leq M \rbrace$ pointwise; 
moreover we can take $M \leq R_X (Cn)$ for some $C>0$. 
Carefully choosing the exact value of $M$ to ensure consistency, 
we use this to construct a local embedding $B_S(n) \rightarrow \Sym(M)$. 

For our lower bound, 
we observe that $\lBrack \sigma \rBrack^{\prime}$ 
contains many copies of the alternating group $\Alt(5)$, 
acting on disjoint subsets of $X$ (hence generating their direct product). 
It follows that any finite group admitting a local embedding 
of a large ball in $\lBrack \sigma \rBrack^{\prime}$ 
also contains a direct product of many copies of $\Alt(5)$ 
as a subgroup, and so has large order. 
The supply of disjoint subsets on which to act, in this construction, 
is limited by the complexity function $p_X$, 
hence the appearance of $p_X$ in Theorem \ref{LBmainthm}. 

This paper is structured as follows: 
in Subsections \ref{LEFSubsect}, \ref{WordsSubsect}, 
and \ref{TFGSubsect} we collect necessary background results about 
LEF growth of groups, symbolic dynamics, and topological full groups, 
respectively. 
In Section \ref{UBSection} we construct the local embeddings 
required to prove Theorem \ref{UBmainthm}. 
In Section \ref{LBSection} we prove Theorem \ref{LBmainthm}. 
In Section \ref{IntSection} we describe the construction of the 
minimal subshifts arising in Theorem \ref{IntGrowthShiftThm}, 
and deduce Theorems \ref{IntGrowthmainthm} and \ref{UnctbleThm}. 

\section{Preliminaries}

\subsection{LEF groups and Schreier graphs} \label{LEFSubsect}

\begin{defn}
For $\Gamma,\Delta$ groups and $F \subseteq \Gamma$, 
a \emph{partial homomorphism} of $F$ into $\Delta$ is a 
function $\phi : F \rightarrow \Delta$ 
such that, for all $g,h \in F$, if $gh \in F$, 
then $\phi(gh)=\phi(g)\phi(h)$. 
A partial homomorphism 
$\phi$ is called a \emph{local embedding} if it is injective. 
$\Gamma$ is \emph{locally embeddable into finite groups (LEF)} 
if, for all finite $F \subseteq \Gamma$, 
there exists a finite group $Q$ and a local embedding of 
$F$ into $Q$. 
\end{defn}

Henceforth suppose that $\Gamma$ is LEF 
and generated by the finite set $S$. 
Let $B_S(n) \subseteq \Gamma$ denote those elements 
of length at most $n$, with respect to the word metric 
induced on $\Gamma$ by $S$. 

\begin{defn}
The \emph{LEF growth} of $\Gamma$ (with respect to $S$) is: 
\begin{center}
$\mathcal{L}_{\Gamma} ^S (n) = \min \lbrace \lvert Q \rvert 
: \exists \phi :B_S(n)\rightarrow Q\text{ a local embedding} \rbrace$
\end{center}
and the \emph{LEF action growth} is: 
\begin{center}
$\mathcal{LA}_{\Gamma} ^S (n) = \min \lbrace d 
: \exists \phi :B_S(n)\rightarrow \Sym(d)\text{ a local embedding} \rbrace$. 
\end{center}
\end{defn}

\begin{rmrk} \label{LEFLAGrthIneqRmrk}
It is clear that 
$\mathcal{LA}_{\Gamma} ^S (n) \leq \mathcal{L}_{\Gamma} ^S (n) 
\leq \mathcal{LA}_{\Gamma} ^S (n)!$. 
\end{rmrk}

\begin{defn} \label{FunctionCompDefn}
For $F_1 , F_2 : \mathbb{N} \rightarrow \mathbb{N}$ nondecreasing functions, 
write $F_1 \preceq F_2$ if there exists $C>0$ 
such that $F_1 (n) \leq F_2 (Cn)$ for all $n$.
Write $F_1 \approx F_2$ if $F_1 \preceq F_2$ and $F_2 \preceq F_1$. 
\end{defn}

\begin{lem}
Let $\mathcal{F} = \mathcal{L}$ or $\mathcal{LA}$. 
Let $\Delta \leq \Gamma$ be finitely generated by $T$. 
Then there exists $C>0$ such that for all $n$, 
\begin{equation*}
\mathcal{F}_{\Delta} ^T (n) \leq \mathcal{F}_{\Gamma} ^S (Cn)
\end{equation*}
In particular, for $T$ a second finite generating set for $\Gamma$, 
$\mathcal{F}_{\Gamma} ^S \approx \mathcal{F}_{\Gamma} ^T$. 
\end{lem}

\begin{proof}
This is proved for $\mathcal{F}=\mathcal{L}$ 
as Corollary 2.7 in \cite{Brad}; 
the proof for $\mathcal{F} = \mathcal{LA}$ is identical. 
\end{proof}

The next Proposition is key to the proof of Theorem \ref{LBmainthm}. 
It uses an idea already exploited in \cite[Theorem 3.4]{Brad} 
to control the LEF growth of wreath products. 

\begin{propn} \label{dirprodLEFprop}
Let $n \in \mathbb{N}$ and $m \geq 2$. 
Suppose $\Delta_1 , \ldots , \Delta_m \leq \Gamma$ 
are finite centreless subgroups, 
generating their direct product, 
and that $\Delta_i \subseteq B_S (n)$. 
Suppose that $Q$ is a finite group and that 
$\phi : B_S (2n) \rightarrow Q$ is a local embedding. 
Then the $\phi (\Delta_i) \leq Q$ generate their direct product, and 
$\lvert Q \rvert \geq \prod_i \lvert \Delta_i \rvert$. 
\end{propn} 

\begin{proof}
Since $\phi$ restricts to an injective homomorphism on each $\Delta_i$, 
$\phi(\Delta_i)$ is a subgroup of $Q$, isomorphic to $\Delta_i$. 
Certainly, for $i \neq j$ and $g_i \in \Delta_i$, $g_j \in \Delta_j$, 
\begin{center}
$\phi(g_i)\phi(g_j) = \phi(g_i g_j) 
= \phi(g_j g_i) = \phi(g_j)\phi(g_i)$. 
\end{center}
Therefore, if the $\phi(\Delta_i)$ fail to generate their direct product, 
there exists $1 \leq i \leq m$ and $1 \neq g \in \Delta_i$ 
such that $\phi (g) \in P$, 
where $P = \langle \phi(\Delta_j):i\neq j \rangle \leq Q$. 
$P$ centralizes $\phi(\Delta_i)$, since the $\phi(\Delta_j)$ do, 
so for $h \in \Delta_i$, 
$\phi (gh) = \phi (g)\phi (h) = \phi (h)\phi (g) = \phi (hg)$. 
By injectivity of $\phi$ restricted to $\Delta_j$, 
$g$ is central in $\Delta_i$, contradiction. 
\end{proof}

A \emph{based graph} is a pair $(G,v)$, 
where $G$ is a directed graph and $v \in V(G)$. 
A \emph{morphism} of based graphs $(G_1,v_1)\rightarrow (G_2,v_2)$ is 
a graph-morphism $\phi: G_1 \rightarrow G_2$ 
with $\phi(v_1)=v_2$. 
For $C$ a set, an \emph{edge-colouring} of the graph $G$ in $C$ 
is a function $c:E(G)\rightarrow C$. 
A morphism $(G_1,v_1,c_1)\rightarrow (G_2,v_2,c_2)$ of based graphs 
with edge colourings in $C$ is a morphism $\phi:(G_1,v_1)\rightarrow (G_2,v_2)$ 
of based graphs such that, for all $e \in E(G_1)$, 
$c_1 (e) = c_2 (\phi (e))$. 

\begin{defn}
Let $\Gamma$ be a group, $\Omega$ be a $\Gamma$-set, and $S \subseteq \Gamma$. 
The associated \emph{Schreier graph} $\Schr(\Gamma,\Omega,\mathbf{S})$ 
is the graph with vertex set $\Omega$ and edge set $\Omega \times S$, 
with the edge $(\omega,s)$ running from $\omega$ to $s \omega$. 
Impose an ordering on the elements of $S$, 
to obtain an ordered $\lvert S \rvert$-tuple 
$\mathbf{S} \in \Gamma ^{\lvert S \rvert}$
(equivalently fix a bijection 
$c:S\rightarrow \lbrace 1,\ldots,\lvert S \rvert \rbrace$). 
Then $\Schr(\Gamma,\Omega,\mathbf{S})$ is naturally 
an edge-coloured graph with colours 
in $\lbrace 1,\ldots,\lvert S \rvert \rbrace$, 
via $c\big( (\omega, s) \big) = c(s)$. 
\end{defn}

\begin{defn}
Let $G_1 , G_2$ be directed edge-coloured graphs 
(with colours in $C$) and let $r \in \mathbb{N}$. 
We say that $G_1$ is 
\emph{locally embedded in $G_2$ at radius $r$} if, 
for every $v \in V(G_1)$, there exists $w \in V(G_2)$ and an 
isomorphism of based coloured graphs $(B_v (r),v) \cong (B_w (r),w)$ 
(here $B_v (r) \subseteq G_1$ is the induced subgraph 
on the closed ball of radius $r$ 
around $v$ in the path metric on $G_1$, and likewise for $B_w (r)\subseteq G_2$).  
We say that $G_1$ and $G_2$ are 
\emph{locally colour isomorphic at radius $r$} if 
each is locally embedded in the other at radius $r$, 
that is, $G_1$ and $G_2$ have the same set of 
isomorphism-types of balls of radius $r$. 
\end{defn}

The next observation is our key tool for constructing the local embeddings 
needed in Theorem \ref{UBmainthm}; 
it is proved as Lemma 4.2 in \cite{Brad}. 

\begin{lem} \label{permisolem}
For $i=1,2$ let $\Gamma_i$ be a group acting faithfully on a set 
$\Omega_i$, and let $\mathbf{S}_i$ 
be an ordered generating $d$-tuple in $\Gamma_i$. 
Suppose that the Schreier graphs 
$\Schr (\Gamma_i,\Omega_i,\mathbf{S}_i)$ 
are locally colour-isomorphic at some radius at least $\lceil 3r/2 \rceil$. 
Then there is a local embedding 
$B_{\mathbf{S}_1} (r) \rightarrow \Gamma_2$ extending 
$(\mathbf{S}_1)_j \mapsto (\mathbf{S}_2)_j$. 
\end{lem}

\subsection{Words} \label{WordsSubsect}

Throughout, $A$ will be a finite discrete set with $\lvert A \rvert \geq 2$, 
the \emph{alphabet}. 
Let $A^{\ast} = \sqcup_n A^n$ be the set of \emph{finite words} over $A$. 
An \emph{infinite word} shall be an element of either $A^{\mathbb{Z}}$ 
or $A^{\mathbb{N}}$. 
We equip these latter sets with the product topology; 
note that both are thereby homeomorphic to the Cantor space. 

\begin{defn} \label{ComplexityDefn}
Let $w \in A^{\mathbb{Z}}$. 
For $k \in \mathbb{Z}$ and $n \in \mathbb{N}$, 
the \emph{$k$th $n$-factor} of $w$ is 
$w_k w_{k+1} \cdots w_{k+n-1} \in A^n$. 
$v \in A^n$ is an $n$-factor of $w$ 
if it is the $k$th $n$-factor for some $k$. 
$n$-factors of words in $A^{\mathbb{N}}$ and $A^{\ast}$ 
are defined similarly, with the requirement that $k \in \mathbb{N}$ and, 
for $w \in A^{\ast}$, that $k \leq \lvert w \rvert-n+1$. 

In any case, the set of all $n$-factors of $w$ is denoted $F_n (w)$. 
For $w$ an infinite word, 
the \emph{complexity function} 
$p_w :\mathbb{N}\rightarrow\mathbb{N}$ of $w$
is given by $p_w (n) = \lvert F_n (w) \rvert$. 
\end{defn}

It is immediate from the definitions that 
$p_w (n+m) \leq p_w(n)p_w(m)$ and 
$p_{w}(n) \leq \lvert A \rvert ^n$ for all $n,m \in \mathbb{N}$. 

\begin{defn}
The \emph{entropy} of $w \in A^{\mathbb{Z}}$ is: 
\begin{equation*}
h(w) = \lim_{n \rightarrow \infty} \frac{\log p_w(n)}{n}
\end{equation*}
(the limit is well-defined, by the preceding remarks). 
\end{defn}

\begin{defn} \label{RecurrenceDefn}
$\mathbf{x} \in A^{\mathbb{Z}}$ is \emph{uniformly recurrent} if, 
for every $n \in \mathbb{N}$, 
there exists $M_n \in \mathbb{N}$ such that, 
for every $w \in F_n (\mathbf{x})$ and $v \in F_{M_n} (\mathbf{x})$, 
$w \in F_n (v)$. 
The smallest such $M_n$ is denoted $R_{\mathbf{x}} (n)$, 
and $R_{\mathbf{x}} : \mathbb{N} \rightarrow \mathbb{N}$ 
is the \emph{recurrence function} of $\mathbf{x}$. 
\end{defn}

Henceforth assume $\mathbf{x} \in A^{\mathbb{Z}}$ 
to be uniformly recurrent non-periodic. 

\begin{thm}[\cite{MorHed} Theorem 7.5] \label{MHThm}
$R_{\mathbf{x}} (n) \geq p_{\mathbf{x}}(n)+n \geq 2n+1$ for all $n \in \mathbb{N}$. 
\end{thm}

In particular, $R_{\mathbf{x}}$ grows at least linearly in $n$. 

\begin{defn}
$\mathbf{x}$ is \emph{linearly recurrent} if there exists 
$C > 0$ such that $R_{\mathbf{x}} (n) \leq Cn$ 
for all $n \in \mathbb{N}$. 
\end{defn}

\begin{ex}[\cite{MorHed2} Theorem 11.4]
The Fibonacci sequence is linearly recurrent.
\end{ex}

\begin{defn}
A \emph{cylinder set} of $A^{\mathbb{Z}}$ is a set of the form: 
\begin{center}
$\lAngle u_{k},\ldots,u_{-1},\underline{u}_0,u_1,\ldots,u_l\rAngle
 = \lbrace \mathbf{y} \in A^{\mathbb{Z}} 
 : y_i = u_i \text{ for } k \leq i \leq l \rbrace$,
\end{center}
where $k,l \in \mathbb{Z}$ with $k \leq 0 \leq l$ and $u_i \in A$ 
for $k \leq i \leq l$. 
For $u \in A^n$ and $1 \leq i \leq n$, 
we write $\lAngle u \rAngle_i$ for the cylinder set 
$\lAngle u_1,\ldots,u_{i-1},\underline{u}_i,u_{i+1},\ldots,u_n\rAngle$. 

For $X \subseteq A^{\mathbb{Z}}$ clopen and $m \in \mathbb{N}$, 
an \emph{$m$-cylinder of $X$} is a nonempty set of the form: 
\begin{center}
$X \cap \lAngle u_{-m},\ldots,u_{-1},\underline{u}_0,u_1,\ldots,u_m\rAngle$. 
\end{center}
\end{defn}

\begin{rmrk} \label{cylpartrmrk}
Let $u \in A^n$ and $1 \leq i \leq n$. 
\begin{itemize}
\item[(i)] For any $k,l\in\mathbb{N}$, 
\begin{equation*}
\bigsqcup_{v \in A^k} \lAngle vu \rAngle_{k+i} = \lAngle u \rAngle_i
 = \bigsqcup_{w \in A^l} \lAngle uw \rAngle_{i}
\end{equation*}
In particular, for any $m \in \mathbb{N}$, the set of all $m$-cylinders 
of $X$ forms a clopen partition of $X$, 
and for $1 \leq k \leq m$, any $k$-cylinder of $X$ is the disjoint union of 
the $m$-cylinders of $X$ which intersect it. 
\item[(ii)] 
The cylinder sets form a basis for the topology on $A^{\mathbb{Z}}$. 
Hence, for any clopen subset $Y$ of $A^{\mathbb{Z}}$, 
$Y$ is the union of cylinder sets. 
By compactness, and by (i), 
there exists $C=C(Y)>0$ such that, 
for all $m \geq C$, $Y$ is the disjoint union 
of $m$-cylinders of $A^{\mathbb{Z}}$.
\end{itemize}
\end{rmrk}

The next Lemma shall be used in the proof of Theorem \ref{IntGrowthShiftThm}. 

\begin{lem} \label{LimitWordLem}
Let $(L_j)$ be an increasing sequence of positive integers; 
let $x^{(j)} \in A^{L_j}$. 
For each $j \in \mathbb{N}$, let $K_j \in \mathbb{N}$ with 
$2 \leq K_j \leq L_{j+1} - L_j$ and suppose that 
$x^{(j)}$ is the $K_j$th $L_j$-factor of $x^{(j+1)}$. 
Choose $M_0 \in \mathbb{N}$ with $1 \leq M_0 \leq L_0$, 
and define $(M_j)$ recursively via $M_{j+1} = M_j + K_j - 1$. 
Then there is a unique point $\mathbf{x} \in A^{\mathbb{Z}}$ 
lying in the intersection of the cylinder sets $\lAngle x^{(j)} \rAngle_{M_j}$. 
Moreover for all $n$, $F_n (\mathbf{x}) = \bigcup_j F_n (x^{(j)})$.
\end{lem}

\begin{proof}
By construction the $\lAngle x^{(j)} \rAngle_{M_n}$ 
form a nested descending sequence of nonempty closed sets 
in the compact space $A^{\mathbb{Z}}$, 
hence their intersection is nonempty. 
Let $\mathbf{x}$ lie in this intersection. 

In $x^{(j)}$, 
there are $M_j - 1 = M_{j-1} + K_{j-1} - 2$ letters strictly to the left of the 
$M_j$th letter, and $L_j - M_j \geq (L_{j-1} - M_{j-1})+1$ 
strictly to the right. Since both these quantities tend to $\infty$, 
for every $i \in \mathbb{Z}$, the $i$th letter of $\mathbf{x}$ 
is uniquely determined by $\mathbf{x} \in \lAngle x^{(j)} \rAngle_{M_j}$, 
provided $j$ is sufficiently large. 
Similarly, for any $i$ and $n$, the $i$th $n$-factor of $\mathbf{x}$ 
lies in $x^{(j)}$ for $j$ sufficiently large, 
so $F_n (\mathbf{x}) \subseteq \bigcup_j F_n (x^{(j)})$. 
Conversely, $\mathbf{x} \in \lAngle x^{(j)} \rAngle_{M_j}$ 
implies $x^{(j)}$ (and hence every $n$-factor of $x^{(j)}$) 
is a factor of $\mathbf{x}$. 
\end{proof}

The utility of Lemma \ref{LimitWordLem} shall be that it allows us 
to an infinite word whose asymptotic features 
(such as the behaviour of the complexity or recurrence functions) 
can be controlled in terms of the sets of factors of the finite words $u^{(n)}$. 

\subsection{Topological full groups and minimal subshifts} \label{TFGSubsect}

Let $\mathbf{C}$ be the Cantor space. 
A \emph{Cantor dynamical system} is a pair $(X,\varphi)$, 
where $X$ is a space homeomorphic to $\mathbf{C}$, 
and $\varphi\in\Homeo(X)$ 
(we specify the space $X$, rather than always taking $X=\mathbf{C}$, 
in case the homeomorphism $\varphi$ 
is described in terms of a particular model $X$ of Cantor space; 
in particular this will be the case when the system is a subshift). 
The system $(X,\varphi)$ is called \emph{minimal} 
if every orbit in $X$ under the action of $\langle \varphi \rangle$ 
is dense in $X$. 

\begin{ex}
Let $A$ be a finite discrete space with $\lvert A \rvert \geq 2$. 
Then $A^{\mathbb{Z}} \cong \mathbf{C}$. 
The \emph{shift} over $A$ is $\sigma \in \Homeo (A^{\mathbb{Z}})$ given, 
for $\mathbf{a} = (a_i)_{i\in \mathbb{Z}}$, 
by $\sigma (\mathbf{a})_i = a_{i+1}$. 
The Cantor dynamical system $(A^{\mathbb{Z}},\sigma)$ is never minimal. 
\end{ex}

\begin{defn}
The \emph{topological full group} $\lBrack \varphi \rBrack$ 
of the system $(X,\varphi)$ 
is the set of all homeomorphisms $g$ of $X$ 
such that there exists a continuous function 
$f_g : X \rightarrow \mathbb{Z}$ 
(called the \emph{orbit cocycle} of $g$) such that 
for all $x \in X$, $g (x) = \varphi ^{f_g (x)} (x)$ 
(here we assume $\mathbb{Z}$ equipped with the discrete topology). 
\end{defn}

Equivalently, $g \in \Homeo (X)$ 
lies in $\lBrack \varphi \rBrack$ if there is 
a finite clopen partition $C_1 , \ldots , C_d$ of $X$ 
and integers $a_1 , \ldots , a_d$ such that for $1 \leq i \leq d$, 
$g|_{C_i} = \varphi^{a_i} |_{C_i}$ 
(taking $\lbrace a_1,\ldots,a_d\rbrace = \im(f_g)$, 
$C_i = f_g ^{-1}(a_i)$). 
It is straightforward to check that $\lBrack \varphi \rBrack$ 
is a subgroup of $\Homeo(X)$. 

\begin{rmrk}
If $(X,\varphi)$ is a minimal system, 
then the orbit cocycle $f_g$ is uniquely determined by 
$g \in \lBrack \varphi \rBrack$, since $\varphi$ has no 
periodic points. 
\end{rmrk}

The next result gives the key source of minimal systems for our purposes. 

\begin{propn} \label{minsubshiftprop}
Suppose $\mathbf{x} \in A^{\mathbb{Z}}$ is uniformly recurrent non-periodic. 
Then $\overline{\mathcal{O}(\mathbf{x})} \cong \mathbf{C}$, 
and the system $(\overline{\mathcal{O}(\mathbf{x})},\sigma)$ is minimal. 
\end{propn}

\begin{proof}
For minimality, see \cite[Theorem~1.0.1]{Jus}. To prove $\overline{\mathcal{O}(\mathbf{x})} \cong \mathbf{C}$, note that the only property not automatically inherited by subspaces of $\mathbf{C}$ is being perfect; then, if $Y=\overline{\mathcal{O}(\mathbf{x})}$, a non-periodic $\mathbf{x}$ must have $|Y|=\infty$, which implies that $Y$ has an accumulation point in $\mathbf{C}$: but $Y$ contains its set of accumulation points $\mathrm{Acc}(Y)$ in $\mathbf{C}$ (by compactness), and $\mathrm{Acc}(Y)$ is closed and $\sigma$-invariant, so by minimality $Y=\mathrm{Acc}(Y)$.
\end{proof}

\begin{defn}
A subspace $X = \overline{\mathcal{O}(\mathbf{x})} \subseteq A^{\mathbb{Z}}$ 
constructed as in the statement of Proposition \ref{minsubshiftprop} 
is called a \emph{minimal subshift}. 
\end{defn}

\begin{rmrk} \label{minseqindeprmrk}
Suppose $\mathbf{x} \in A^{\mathbb{Z}}$ is uniformly recurrent non-periodic. 
For any $\mathbf{y}\in\overline{\mathcal{O}(\mathbf{x})}$ and $n\in\mathbb{N}$, 
$F_n (\mathbf{x}) = F_n (\mathbf{y})$. 
In particular, $p_{\mathbf{x}}$ and $R_{\mathbf{x}}$ depend only on $X$, 
and we may henceforth write $p_X$ or $R_X$ instead. 
\end{rmrk}

Let $\lBrack \varphi \rBrack^{\prime}$ denote 
the derived subgroup of $\lBrack \varphi \rBrack$. 
The reason for our focus on minimal subshifts, 
among all minimal systems, is made clear by the next result. 

\begin{thm}[\cite{Mat} Theorem 5.4] \label{Matuithm}
For $(X,\varphi)$ a Cantor minimal system, 
$\lBrack \varphi \rBrack^{\prime}$ is a finitely generated group 
iff $(X,\varphi)$ is isomorphic to a minimal subshift. 
\end{thm}

\begin{thm}[\cite{GrigMedy} Proposition 2.4]
Let $(X,\sigma)$ be a minimal subshift, 
and let $S$ be a finite generating set for $\lBrack \sigma \rBrack^{\prime}$. 
Then $\lvert B_S(n) \rvert \succeq \exp (n)$. 
\end{thm}

\begin{coroll}
We have $\mathcal{L}_{\lBrack \sigma \rBrack^{\prime}}(n) \succeq \exp(n)$. 
\end{coroll}

\begin{proof}
This is immediate from the preceding Theorem: 
if $\phi: B_S (n) \rightarrow Q$ is a local embedding, 
then $\phi$ is injective, so $\lvert Q \rvert \geq \lvert B_S(n) \rvert$. 
\end{proof}

Thus our Theorem \ref{LBmainthm} is only new in 
the case $p_X (n) \npreceq n^2$. 

\section{Construction of local embeddings} \label{UBSection}

In this Section we prove Theorem \ref{UBmainthm}. 
Let $(X,\sigma)$ be a minimal subshift over the alphabet $A$. 

\begin{propn} \label{linboundpropn}
Let $S \subseteq \lBrack\sigma\rBrack$ be finite. 
Then there exists an integer 
$C_1 = C_1(S) \geq 1$ such that for all 
$r \in \mathbb{N}$ and all $g \in B_S (r)$: 
\begin{itemize}
\item[(i)] $\max\lbrace\lvert f_g (x)\rvert : x\in X\rbrace\leq C_1 r$; 

\item[(ii)] For all $m \geq C_1 r$, 
$f_g$ is constant on every $m$-cylinder of $X$. 
\end{itemize}
\end{propn}

\begin{proof}
For (i), for $g \in \lBrack\sigma\rBrack$ write 
$\lambda (g) = \max\lbrace\lvert f_g (x)\rvert : x\in X\rbrace$. 
Since $S$ is finite, we may choose 
$C_1 \geq \max \lbrace \lambda(s):s\in S \rbrace$. 
For $g,h\in \lBrack\sigma\rBrack$, 
$\lambda(gh) \leq \lambda(g)+\lambda(h)$, 
so by induction $\lambda(g) \leq C_1 r$ for all $g \in B_S (r)$. 

For (ii), for every $s \in S \cup S^{-1}$ there is a finite 
clopen partition $\mathcal{C}_s$ of $X$, 
such that $f_s$ is constant on the parts of $\mathcal{C}_s$. 
By Remark \ref{cylpartrmrk} (ii), we may choose $C_1$ sufficiently 
large that for all $s \in S \cup S^{-1}$ and $m \geq C_1$, 
$f_s$ is constant on every $m$-cylinder of $X$. 
Let $r \geq 2$ and suppose by induction that the claim 
holds for smaller $r$. Let $m \geq C_1 r$, 
let $U$ be an $m$-cylinder of $X$, 
and let $g \in B_S (r)$. 
Then there exist $h\in B_S (r-1)$ and $s\in S\cup S^{-1}\cup \lbrace e\rbrace$ 
such that $g=sh$. 
By inductive hypothesis, $f_h$ is constant on $U$, say with value $i$.
Observe that, for any $k,l \in \mathbb{N}$, 
if $x,y \in X$ lie in the same $(k+l)$-cylinder of $X$, 
then $\sigma^{\pm l} x , \sigma^{\pm l} y$ 
lie in the same $k$-cylinder of $X$.
Thus, there is a unique $(m-|i|)$-cylinder
containing $h(U)$. By (i), $i \in [-C_1 (r-1) , C_1 (r-1)]$, so $m-|i|\geq C_{1}$:
this implies that $f_s$ is constant on $h(U)$, 
hence $f_g$ is constant on $U$. 
\end{proof}

\begin{rmrk} \label{shiftcylindersrmrk}
By Proposition~\ref{linboundpropn} (i), for $k\in\mathbb{N}$ and $m \geq C_1 r$, 
if $x,y \in X$ lie in the same $(k+m)$-cylinder of $X$, 
then $g(x),g(y)$ lie in the same $k$-cylinder of $X$ 
for any $g \in B_S (r)$. 
\end{rmrk}

\begin{lem} \label{intervalconstrlem}
Let $C_1$ be as in Proposition \ref{linboundpropn}. 
For all $r \in \mathbb{N}$, 
there exists $M = M(r) \in \mathbb{N}$ satisfying: 
\begin{itemize}
\item[(i)] $\lbrace \sigma^i \mathbf{x} : 1\leq i\leq M \rbrace$ 
intersects every $C_1(r+1)$-cylinder of $X$; 
\item[(ii)] $\mathbf{x}$ and $\sigma^M \mathbf{x}$ lie in the same 
$C_1(2r+1)$-cylinder of $X$; 
\item[(iii)] $10 C_1 r \leq M \leq 2 R_X (10 C_1 r)$. 
\end{itemize}
\end{lem}

\begin{proof}
Let $u \in A^{2C_1(r + 1)+1}$. 
$\lAngle u \rAngle_{C_1(r + 1)+1}\cap X$ 
is a $C_1(r + 1)$-cylinder of $X$ iff 
$u \in F_{2C_1(r + 1)+1} (\mathbf{x})$. 
Therefore (i) holds for any $M \geq R_X (2C_1(r + 1)+1)$. 
For the same reason, a value of $M$ satisfying (ii) 
occurs at least once in every interval 
of length $R_X (2 C_1(2r+1)+1)$. 
Therefore we can choose $M \leq R_X (2 C_1(2r+1)+1) + 10 C_1 r$ 
satisfying (i), (ii) and the first inequality of (iii). 
Since $10 C_1 r \leq R_X (5 C_1 r)$ by Theorem \ref{MHThm}, 
and since $R_X$ is nondecreasing, 
the second inequality of (iii) also holds. 
\end{proof}

\begin{coroll} \label{intervalendptscor}
Let $M$ be as in Lemma \ref{intervalconstrlem}. 
For any $\lvert i \rvert \leq C_1 r$, 
$\sigma^i \mathbf{x}$ and $\sigma^{M+i} \mathbf{x}$ lie in the same 
$C_1(r+1)$-cylinder of $X$. 
\end{coroll}

\begin{proof}
This is immediate from Lemma \ref{intervalconstrlem} (ii), 
by Remark \ref{shiftcylindersrmrk}. 
\end{proof}

Fix $\mathbf{x} \in X$ and let $\mathcal{O}=\mathcal{O}(\mathbf{x})$ 
be the orbit of $\mathbf{x}$ under $\sigma$. 
Then $\lBrack \sigma \rBrack$ acts on $\mathcal{O}$. 
Since $X$ has no periodic points, 
this action induces a well-defined homomorphism 
$\phi : \lBrack \sigma \rBrack \rightarrow \Sym (\mathbb{Z})$ by: 
\begin{equation}
g\big(\sigma^n \mathbf{x}\big)=\sigma^{\phi(g)[n]} \big( \mathbf{x}\big)
\text{.}
\end{equation}
In other words, $\phi(g) [n]  = n + f_{g}(\sigma^n \mathbf{x})$. 
Since $\mathcal{O}$ is dense in $X$, $\phi$ is injective. 

Let $\Gamma \leq \lBrack \sigma \rBrack$ be finitely generated, 
and let $\mathbf{S} = (s_1,\ldots,s_d)$ be an ordered 
generating $d$-tuple of elements of $\Gamma$. 
Then $\Gamma$ acts faithfully on $\mathbb{Z}$ via $\phi$. 
Let $G = \Schr(\Gamma,\mathbb{Z},\mathbf{S})$ be the associated Schreier graph 
of the action of $\Gamma$ on $\mathbb{Z}$ with respect to $\mathbf{S}$. 
Let $C_1 = C_1 (\lbrace s_1 , \ldots,s_d \rbrace)\geq 1$ 
be as in Proposition \ref{linboundpropn}. 

\begin{lem} \label{BallSchrCylLem1}
For $n \in \mathbb{Z}$ and $r \in \mathbb{N}$, 
the isomorphism type of $B_G(n,r) \subseteq G$ 
(as a based, edge-coloured graph) 
depends only on the $C_1 (r + 1)$-cylinder of $X$ 
in which $\sigma^n \mathbf{x}$ lies. 
\end{lem}

\begin{proof}
By Proposition \ref{linboundpropn} (i), 
$B_G(n,r) \subseteq [n-C_1 r , n+C_1 r] \cap \mathbb{Z}$. 
For $\lvert i\rvert \leq C_1 r $, $n+i \in B_G(n,r)$ iff, 
for some $g \in B_{\mathbf{S}} (r)$, 
$f_g(\sigma^n \mathbf{x}) = i$, 
and by Proposition \ref{linboundpropn} (ii), 
this condition depends only on the $C_1 r$-cylinder of $X$ 
containing $\sigma^n \mathbf{x}$. 

For $n+i,n+j \in B_G(n,r)$, and $1 \leq c \leq d$, 
$(n+i,n+j)$ is a $c$-coloured edge in $G$ iff 
$f_{s_c} (\sigma^{n+i} \mathbf{x}) = j-i$. 
This depends only on the $C_1$-cylinder of $X$ 
containing $\sigma^{n+i} \mathbf{x}$, 
which in turn depends 
only on the $C_1 (r+1)$-cylinder of $X$ containing $\sigma^n \mathbf{x}$. 
\end{proof}

Let $r \geq 1$ and let $M = M(r)$ be as in Lemma \ref{intervalconstrlem}. 
For $1\leq c \leq d$ define 
$\overline{s}_c :\mathbb{Z}/M\mathbb{Z} \rightarrow \mathbb{Z}/M\mathbb{Z}$ by: 
\begin{center}
$\overline{s}_c (n + M \mathbb{Z}) = \phi(s_c)[n] + M \mathbb{Z}$ 
for $1 \leq n \leq M$. 
\end{center}

\begin{propn} \label{SymPropn}
The function $\overline{s}_c$ lies in $\Sym(\mathbb{Z}/M\mathbb{Z})$. 
\end{propn}

\begin{proof}
We check that $\overline{s}_c$ is injective, 
and therefore is indeed a well-defined permutation. 
If $1 \leq m<n \leq M$ with $\phi(s_c)[n] \equiv \phi(s_c)[m] \mod M$, 
then: 
\begin{center}
$n-m\equiv f_{s_c}(\sigma^m \mathbf{x})-f_{s_c}(\sigma^n\mathbf{x})\mod M$
\end{center}
but by Proposition \ref{linboundpropn} (i), 
\begin{center}
$\lvert f_{s_c}(\sigma^m \mathbf{x})-f_{s_c}(\sigma^n\mathbf{x}) \rvert \leq 2 C_1$
\end{center}
while $1 \leq n-m \leq M-1$, so either $1 \leq n-m \leq 2C_1$ 
or $M-2C_1 \leq n-m \leq M-1$. 
In the former case, 
\begin{center}
$\phi(s_c)[n] \equiv \phi(s_c)[m] \mod M$ 
but $\lvert \phi(s_c)[n] - \phi(s_c)[m] \rvert \leq 4 C_1$,
\end{center}
and since $M \geq 10 C_1$ by Lemma \ref{intervalconstrlem} (iii), 
$\phi(s_c)[n]=\phi(s_c)[m]$, contradicting the
injectivity of $\phi(s_c)$. 
The latter case is similar; we have: 
\begin{center}
$1 \leq m \leq 2C_1$ and $1 \leq m+M - n \leq 2C_1$. 
\end{center}
By Corollary \ref{intervalendptscor}, 
$\sigma^{m+M} \mathbf{x}$ and $\sigma^m \mathbf{x}$ 
lie in the same $C_1 (r+1)$-cylinder of $X$, 
so by Proposition \ref{linboundpropn}, 
$f_{s_c} (\sigma^{m+M} \mathbf{x}) = f_{s_c} (\sigma^m \mathbf{x})$. 
We may therefore argue as in the former case, 
with $m+M$ replacing $m$. 
\end{proof}

Theorem \ref{UBmainthm} is now immediate, by Remark \ref{LEFLAGrthIneqRmrk}, 
from the next result. 

\begin{thm}
We have 
$\mathcal{LA}_{\Gamma} ^S (\lfloor 2r/3 \rfloor) \leq 2 R_X (10 C_1 r)$. 
\end{thm}

\begin{proof}
Let $M$ be as in Lemma \ref{intervalconstrlem}. 
We claim that there is a local embedding $B_{\mathbf{S}}(\lfloor 2r/3 \rfloor) 
\rightarrow \Sym (\mathbb{Z}/M\mathbb{Z})$ 
sending $s_c$ to $\overline{s}_c$ for $1 \leq c \leq d$. 
By the upper bound on $M$ from Lemma \ref{intervalconstrlem} (iii), 
the result follows. 
Let $\overline{\mathbf{S}}=\lbrace\overline{s}_1 ,\ldots ,\overline{s}_d\rbrace$; 
let $\overline{\Gamma} = \langle \overline{\mathbf{S}} \rangle 
\leq \Sym (\mathbb{Z}/M\mathbb{Z})$, 
and let $\overline{G} = \Schr (\overline{\Gamma},\mathbb{Z}/M\mathbb{Z},\overline{\mathbf{S}})$. 
By Lemma \ref{permisolem}, it suffices to show that 
$G$ and $\overline{G}$ are locally colour-isomorphic at radius $r$. 

By Lemma \ref{BallSchrCylLem1} and Lemma \ref{intervalconstrlem} (i), 
for every $m \in \mathbb{Z}$, 
$B_G(m , r) \subseteq G$ is 
isomorphic (as a based, edge-coloured graph) to $B_G(n , r)$, 
for some $1 \leq n \leq M$. 
Let $\pi_M:\mathbb{Z}\rightarrow\mathbb{Z}/M\mathbb{Z}$ be reduction modulo $M$. 
It suffices to check that for every $1 \leq n \leq M$ 
the restriction of $\pi_M$ to $V\big( B_G(n , r) \big) \subseteq \mathbb{Z}$ 
induces an isomorphism of based edge-coloured graphs 
$B_G(n,r) \rightarrow B_{\overline{G}}(\pi_M (n),r)$.  

Consider $I_n = [n-C_1 r,n+C_1 r] \cap \mathbb{Z}$.
Since $2 C_1 r < M$ (by Lemma \ref{intervalconstrlem}), 
the restriction of $\pi_M$ to $I_n$ is injective. 
By Proposition \ref{linboundpropn} (i), 
$V\big( B_G(n , r) \big) \subseteq I_{n}$, 
so the restriction of $\pi_M$ to $V\big( B_G(n , r) \big)$ 
is a bijection onto its image. Then it suffices to show that
$\pi_{M}(\phi(s_{c})[k])=\overline{s}_{c}(\pi_{M}(k))$ for all $k\in I_{n}$ and all $c$.

Since $-C_1 r \leq k \leq M + C_1 r$, there exists $\epsilon\in\{0,\pm 1\}$
such that $k'=k+\epsilon M$ is the representative of $k+M\mathbb{Z}$ in $[1,M]$.
By Corollary~\ref{intervalendptscor} and Proposition~\ref{linboundpropn} (ii)
(or by $k=k'$) we have $f_{s_{c}}(\sigma^{k}\mathbf{x})=f_{s_{c}}(\sigma^{k'}\mathbf{x})$, which means $\phi(s_{c})[k]=\phi(s_{c})[k']+\epsilon M$. Then,
\begin{equation*}
\pi_{M}(\phi(s_{c})[k])=\pi_{M}(\phi(s_{c})[k']+\epsilon M)=\pi_{M}(\phi(s_{c})[k'])=\overline{s}_{c}(k'+M\mathbb{Z})=\overline{s}_{c}(\pi_{M}(k))
\end{equation*}
as claimed. 

By the preceding paragraph, the restriction of $\pi_M$ to $V\big( B_G(n , r) \big)$ 
preserves edges and colours. 
In particular, the vertices of $B_{\overline{G}}(\pi_M (n),r)$, 
which are precisely the endpoints of (undirected) edge-paths of length $\leq r$ 
in $\overline{G}$ starting at $\pi_M (n)$, 
are exactly the image under $\pi_M$ of the endpoints 
of edge-paths of length $\leq r$ in $G$ starting at $n$, 
namely the vertices of $B_G(n , r)$. 
Hence $\pi_M$ does indeed induce the desired isomorphism of based 
edge-coloured graphs. 
\end{proof}

\section{Obstructions to small local embeddings} \label{LBSection}

Let $(X,\sigma)$ be a minimal subshift over $A$, 
and fix $\mathbf{x} \in A^{\mathbb{Z}}$ 
with $X = \overline{\mathcal{O}(\mathbf{x})}$. 
Let $S \subseteq \lBrack\sigma\rBrack^{\prime}$ be a finite generating set 
(such exists by Theorem \ref{Matuithm}), 
and let $\lvert \cdot \rvert_S$ be the word length function induced 
on $\lBrack\sigma\rBrack^{\prime}$ by $S$. 
Recall that 
$p_X = p_{\mathbf{x}} : \mathbb{N} \rightarrow \mathbb{N}$ 
is the \emph{complexity function} of $\mathbf{x}$ 
and $R_X = R_{\mathbf{x}} : \mathbb{N} \rightarrow \mathbb{N}$ 
is the \emph{recurrence function} of $\mathbf{x}$. 

\begin{lem} \label{numcyllem}
Let $\Cyl_X (m)$ be the set of $m$-cylinders of $X$. 
Then $\lvert \Cyl_X (m) \rvert = p_X (2m+1)$. 
\end{lem}

\begin{proof}
$\mathcal{O}(\mathbf{x})=\lbrace\sigma^i(\mathbf{x}):i\in\mathbb{Z}\rbrace$ 
is dense in $X$, so for $u_i \in A$, 
$U = \lAngle u_{-m},\ldots,u_{-1},\underline{u}_0,u_1,\ldots,u_m\rAngle$ 
intersects $X$ iff for some $i \in \mathbb{Z}$, 
$\sigma^i (\mathbf{x}) \in U$, which occurs iff 
$u_j = x_{i+j}$ for $-m \leq j \leq m$. 
That is, $U$ intersects $X$ iff $u_{-m} \cdots u_m$ 
is a factor of $\mathbf{x}$. 
\end{proof}

\begin{lem} \label{disjointcyllem}
Suppose that $m \geq (R_{\mathbf{x}}(4)-1)/2$, and let $U \in \Cyl_X(m)$. 
Then the sets $\sigma^{i}(U)$, for $-2\leq i\leq 2$, are pairwise disjoint. 
\end{lem}

\begin{proof}
Suppose to the contrary that $\mathbf{y} \in \sigma^{i}(U) \cap \sigma^{j}(U)$, 
for some $-2 \leq i < j \leq 2$. Writing $k =j-i$, 
we have $y_l = y_{l+k}$ for $-m \leq l \leq m-k$. 
Letting $w = y_{-m} y_{-m+1}\cdots y_m \in F_{2m+1} (\mathbf{x})$, 
we have $\lvert F_k(w)\rvert \leq k$. 
On the other hand, $2m+1 \geq R_{\mathbf{x}}(k)$, 
so by the definition of $R_{\mathbf{x}}$, 
$\lvert F_k(w) \rvert = \lvert F_k(\mathbf{x}) \rvert =p_{\mathbf{x}}(k)\geq k+1$ 
(the last inequality holding by Theorem \ref{MHThm}), a contradiction. 
\end{proof}

\begin{lem} \label{disjoint5seqlem}
Let $m \geq \max \lbrace (R_{\mathbf{x}}(4)-1)/2 , 5 \rbrace$. 
There exists a set $\DCyl_X(m) \subseteq \Cyl_X(m)$ such that
the sets $\sigma^i (U)$, for $ U \in \DCyl_X(m)$ and $-2 \leq i \leq 2$, 
are pairwise-disjoint, and: 
\begin{equation} \label{DCylbig}
\lvert \DCyl_X(m) \rvert \geq p_{\mathbf{x}} (2m-7)/9 \text{.}
\end{equation}
\end{lem}

\begin{proof}
We construct $\DCyl_X(m)$ via an iterative process, as follows. 
At step $0$, let $U_0$ be any $m$-cylinder. 
Then the $\sigma^i (U_0)$, for $-2 \leq i \leq 2$, 
are pairwise disjoint, by Lemma \ref{disjointcyllem}. 
Let $A_0$ be the set of $(m-4)$-cylinders containing one of the $\sigma^j (U_0)$, 
for $-4 \leq j \leq 4$, 
so that $\lvert A_0 \rvert \leq 9$, 
and let $B_0 = \lbrace U_0 \rbrace$. 

At each subsequent step $k+1$, 
we start with $B_k$ a family of $m$-cylinders 
and $A_k$ a family of $(m-4)$-cylinders such that, 
for any $U \in B_k$ and $-4 \leq j \leq 4$, 
there exists $V \in A_k$ such that $\sigma^j (U) \subseteq V$. 
If $A_k$ is a cover of $X$, 
then set $\DCyl_X(m) = B_k$ and stop. 
Otherwise, choose $\mathbf{y}_k \in X \setminus (\bigcup_{V \in A_k} V)$ and 
let $U_{k+1}$ be the $m$-cylinder of $X$ containing $\mathbf{y}_k$. 
Then $\sigma^i (U_{k+1})$ ($-2 \leq i \leq 2$) are pairwise disjoint 
(by Lemma \ref{disjointcyllem}). 
If there exists $0 \leq l \leq k$ and $-2 \leq i,j \leq 2$ 
such that $\sigma^i (U_{k+1}) \cap \sigma^j (U_l) \neq \emptyset$, 
then, letting $V \in A_k$ be such that $\sigma^{j-i} (U_l) \subseteq V$, 
$U_{k+1} \subseteq V$ (by Remark \ref{cylpartrmrk}), 
contradicting the choice of $\mathbf{y}_k$. 

We may therefore let $B_{k+1} = B_k \cup \lbrace U_{k+1} \rbrace$ 
and produce $A_{k+1}$ by adding to $A_k$ all $(m-4)$-cylinders 
containing one of the $\sigma^j (U_{k+1})$ ($-4 \leq j \leq 4$). 
At every stage $\lvert A_{k+1} \rvert \leq \lvert A_k \rvert +9$, 
$\lvert B_{k+1} \rvert \geq \lvert B_k \rvert + 1$, 
and by Lemma \ref{numcyllem}, 
the process terminates only when $\lvert A_k \rvert = p_{\mathbf{x}} (2m-7)$. 
\end{proof}

\begin{notn}
Let $U \subseteq X$ be a nonempty clopen set, 
such that $\sigma^{-1}(U)$, $U$ and $\sigma(U)$ are pairwise disjoint. 
We denote by $f_U$ the element of $\lBrack \sigma \rBrack$ given by: 
\begin{equation*}
f_U (\mathbf{y}) =\begin{cases}
\sigma(\mathbf{y}) & \mathbf{y} \in U \cup \sigma^{-1} (U), \\
\sigma^{-2}(\mathbf{y}) & \mathbf{y} \in \sigma (U), \\
\mathbf{y} & \text{otherwise.}
\end{cases}
\end{equation*}
\end{notn}

\begin{lem} \label{fUinCommSubgrpLem}
All $f_U$ lie in $\lBrack \sigma \rBrack^{\prime}$. 
\end{lem}

\begin{proof}
Define $h_U \in \lBrack \sigma \rBrack$ by: 
\begin{equation*}
h_U (\mathbf{y}) =\begin{cases}
\sigma(\mathbf{y}) & \mathbf{y} \in \sigma^{-1} (U), \\
\sigma^{-1}(\mathbf{y}) & \mathbf{y} \in U, \\
\mathbf{y} & \text{otherwise.}
\end{cases}
\end{equation*}
Then $f_U = f_U ^{-1} h_U ^{-1} f_U h_U$. 
\end{proof}

The following identities appear as \cite[Lemma 5.3]{Mat}; 
they are proved by direct calculation, 
some of which is explained in \cite[Lemma 3.0.11]{Jus}. 

\begin{lem} \label{ConjCommIdLem}
Let $U , V \subseteq X$ be nonempty clopen subsets. 
\begin{itemize}
\item[(i)] If $\sigma^i (V)$ are pairwise disjoint for $-2 \leq i \leq 2$, 
and $U \subseteq V$, then: 
\begin{align*}
\tau_V f_U \tau_V ^{-1} & = f_{\sigma(U)} \\
\tau_V ^{-1} f_U \tau_V & = f_{\sigma^{-1}(U)}
\end{align*} 
where $\tau_V = f_{\sigma^{-1}(V)} f_{\sigma(V)}$; 
\item[(ii)] If $\sigma^{-1} (U)$, $U$, $\sigma (U) \cup \sigma^{-1} (V)$, 
$V$, $\sigma(V)$ are pairwise disjoint, then: 
\begin{equation*} 
f_{\sigma (U) \cap \sigma^{-1}(V)} = f_V f_U ^{-1} f_V ^{-1} f_U\text{.}
\end{equation*}
\end{itemize}
\end{lem}

\begin{propn} \label{gen3-cycprop}
There exists $C_2 = C_2 (X,S) > 0$ such that 
for all $m \geq (R_{\mathbf{x}}(4)-1)/2$, 
if $W \subseteq X$ is an $m$-cylinder of $X$, 
then $f_{\sigma^{-1}(W)},f_W,f_{\sigma(W)} \in B_S (C_2 m^2)$. 
\end{propn}

\begin{proof}
Write $C_0 = \lceil (R_{\mathbf{x}}(4)-1)/2 \rceil+1$, 
and define $m_n = C_0 (2^n + 1)$, so that $m_{n+1} = 2 m_n - C_0$. 
We inductively construct a nondecreasing sequence $(l_n)$ 
of positive integers, 
such that, if $W \subseteq X$ is an $m$-cylinder of $X$, 
with $C_0 \leq m \leq m_n$, 
then $f_{\sigma ^{-1}(W)} , f_W , f_{\sigma(W)} \in B_S (l_n)$. 
We then analyze the growth of $l_n$. 
Since there are finitely many $m$-cylinders $W$ of $X$ 
with $C_0 \leq m \leq m_0 = 2C_0$, and for each, 
$f_W \in \lBrack \sigma \rBrack^{\prime}$ by Lemma \ref{fUinCommSubgrpLem}, 
there is a constant $l_0$ such that $f_{\sigma ^{-1}(W)} , f_W , f_{\sigma(W)} \in B_S (l_0)$ for all such $W$. 
For $n \geq 1$, suppose $m_{n-1} < m \leq m_n$ and let: 
\begin{equation*}
W = \lAngle w_{-m},\ldots,w_{-1},\underline{w_0},w_1,\ldots,w_m \rAngle \cap X
\end{equation*}
be an $m$-cylinder of $X$. Set $C = C_0$ or $C_0-1$, 
such that $m + C$ is even, and let: 
\begin{align*}
U & = \lAngle w_{-m},\ldots,w_{-2},\underline{w_{-1}},w_0,\ldots,w_{C} \rAngle \\
V & = \lAngle w_{-C},\ldots,w_0,\underline{w_1},w_2,\ldots,w_m \rAngle
\end{align*}
so that $\sigma (U) \cap \sigma^{-1} (V) = W$, and:
\begin{center}
$\sigma (U) \cup \sigma^{-1} (V) 
\subseteq \lAngle w_{-C},\ldots,w_{-1},\underline{w_0},w_1,\ldots,w_C \rAngle$, 
\end{center}
so by Lemma \ref{disjointcyllem}, 
$U$ and $V$ satisfy the conditions of Lemma \ref{ConjCommIdLem} (ii), and: 
\begin{equation} \label{CommLengthBound}
\lvert f_W \rvert_S \leq 2 (\lvert f_U \rvert_S + \lvert f_V \rvert_S)\text{.}
\end{equation}
Now, $\sigma^{-(C+m-2)/2}(U)$ and $\sigma^{(C+m-2)/2}(V)$ are 
$(m+C)/2$-cylinders of $X$, and $(m+C)/2 \leq m_{n-1}$, so by induction, 
\begin{equation} \label{ShiftLengthBound}
\lvert f_{\sigma^{-(C+m-2)/2}(U)}\rvert_S,
\lvert f_{\sigma^{(C+m-2)/2}(V)}\rvert_S \leq l_{n-1}\text{.}
\end{equation}
Let $U_i ^{\prime}$ be the $C_0$-cylinder of $X$ containing 
$\sigma^{-i} (U)$, for $1 \leq i \leq (C+m-2)/2$. 
Then by Lemma \ref{ConjCommIdLem} (i), $f_U$ can be obtained 
from $f_{\sigma^{-(C+m-2)/2}(U)}$ by conjugating by 
$\tau_{U_{(C+m-2)/2} ^{\prime}} , \ldots , \tau_{U_1 ^{\prime}}$ 
in sequence. 
By our base case, $\lvert \tau_{U_i ^{\prime}} \rvert_S \leq 2 l_0$ 
for all $i$, so
\begin{equation*}
\lvert f_U\rvert_S \leq \lvert f_{\sigma^{-(C+m-2)/2}(U)}\rvert_S 
+ 2l_0 (C+m-2)
\end{equation*}
and arguing similarly for $V$, 
\begin{equation*}
\lvert f_V \rvert_S \leq \lvert f_{\sigma^{(C+m-2)/2}(V)}\rvert_S
+ 2l_0 (C+m-2)\text{.}
\end{equation*}
Combining with (\ref{CommLengthBound}), (\ref{ShiftLengthBound}), 
and our bound on $m_n$, we have: 
\begin{equation} \label{FinalLengthBound}
\lvert f_W \rvert_S \leq 4 l_{n-1} + C^{\prime} 2^n + C^{\prime\prime}
\end{equation}
for some constants $C^{\prime},C^{\prime\prime} > 0$. 
Finally, applying Lemma \ref{ConjCommIdLem} (i) a final time, 
we conjugate $f_W$ by $\tau_{W^{\prime}}$ or $\tau_{W^{\prime}} ^{-1}$ 
where $W^{\prime}$ is the $C_0$-cylinder of $X$ containing $W$, 
and $\lvert f_{\sigma(W)} \rvert_S , \lvert f_{\sigma^{-1}(W)} \rvert_S$ 
also satisfy a bound as in (\ref{FinalLengthBound}) 
(for some larger $C^{\prime\prime}$). 
We may therefore take $l_n = 4 l_{n-1} + C^{\prime} 2^n + C^{\prime\prime}$, 
so that $f_{\sigma ^{-1}(W)} , f_W , f_{\sigma(W)} \in B_S (l_n)$. 
Solving the recurrence for $l_n$ we obtain 
$l_n \leq C_2 ^{\prime} 4^n \leq (4 C_2 ^{\prime} / C_0 ^2) m_{n-1} ^2 \leq (4 C_2 ^{\prime} / C_0 ^2) m^2$ for some $C_2 ^{\prime} > 0$. 
\end{proof}

Recall that for $g \in \Homeo (X)$, the \emph{support} of $g$ is: 
\begin{center}
$\supp(g) = \lbrace x \in X : g(x)\neq x \rbrace$. 
\end{center}

\begin{propn} \label{keyLBprop}
Let $m$ and $\DCyl_X(m)$ be as in Lemma \ref{disjoint5seqlem}. 
There exists $C_3 = C_3(X,S)>0$ such that, for all $U \in \DCyl_X(m)$, 
there is a subgroup $\Delta_U \leq \lBrack\sigma\rBrack^{\prime}$ satisfying: 
\begin{itemize}
\item[(i)] $\Delta_U \cong \Alt(5)$; 
\item[(ii)] $\Delta_U \subseteq B_S (C_3 m^2)$; 
\item[(iii)] For all $g \in \Delta_U$, 
\begin{equation} \label{Alt5SuppEqn}
\supp(g) \subseteq \bigcup_{i=-2} ^2 \sigma^i (U)\text{.}
\end{equation}
\end{itemize}
\end{propn}

\begin{proof}
For $U \in \DCyl_X (m)$, let 
$\Delta_U = \langle f_{\sigma^{-1}(U)},f_U,f_{\sigma(U)} \rangle$. 
Then (iii) holds, since it holds 
for $g \in \lbrace f_{\sigma^{-1}(U)},f_U,f_{\sigma(U)} \rbrace$. 
Now $\Delta_U$ acts, faithfully, on 
$\lbrace\sigma^i(U):\lvert i\rvert\leq 2\rbrace$. 
Identifying this set in the obvious way with $\lbrace -2,-1,0,1,2 \rbrace$,  
$f_{\sigma^{-1}(U)}$,$f_U$ and $f_{\sigma(U)}$ act, respectively, 
as the $3$-cycles $(-2 \: -1 \: 0)$, $(-1 \: 0 \: 1)$ and $(0 \: 1 \: 2)$. 
As such, $\Delta_U$ acts as the alternating group 
on $\lbrace -2,-1,0,1,2\rbrace$, 
and we have (i). 

By Proposition \ref{gen3-cycprop}, 
$f_{\sigma^{-1}(U)},f_U,f_{\sigma(U)} \in B_S (C_2 m^2)$, 
and since $\lvert \Delta_U \rvert = 60$, 
$\Delta_U \subseteq B_S (60 C_2 m^2)$, whence (ii). 
\end{proof}

\begin{proof}[Proof of Theorem \ref{LBmainthm}]
Let $Q$ be a finite group and $\pi : B_S (r) \rightarrow Q$ 
be a local embedding. 
For any $m \leq (r/2C_3)^{1/2}$, we have $\Delta_U \leq B_S(r/2)$, 
where $C_3 > 0$ and $\Delta_U$ are as in Proposition \ref{keyLBprop}. 
We apply Proposition \ref{dirprodLEFprop} to the family 
$\lbrace \Delta_U : U \in \DCyl(m) \rbrace$ for $m \geq c' r^{1/2}$ 
($c'>0$ sufficiently small), and $r$ larger than a constant 
such that $m \geq (R_{\mathbf{x}}(4)-1)/2$. 
Since the $\Delta_U$ are disjointly supported 
(by Lemma \ref{disjoint5seqlem} and (\ref{Alt5SuppEqn})), 
they do indeed generate their direct product. 
From Proposition \ref{dirprodLEFprop} we conclude: 
\begin{align*}
\lvert Q \rvert & \geq \prod_{U \in \DCyl_X(m)} \lvert \Delta_U \rvert & \\
& \geq \exp (c p_{\mathbf{x}} (2m-7))&  \text{ (by (\ref{DCylbig}))} \\
& \geq \exp (c p_{\mathbf{x}} (2c' r^{1/2}-7))
\end{align*}
for $c = \log(60)/9$. 
\end{proof}

\begin{rmrk}
\normalfont
We can improve the constant $c$ by modifying the construction so as to take 
$\Delta_U \cong \Alt(2d+1)$ for large $d$, instead of $\Alt(5)$. 
To do this we would need to take $\DCyl_X (m)$ to be a family 
of $m$-cylinders $U$ such that all sets $\sigma^i (U)$ 
are pairwise disjoint for $-d \leq i \leq d$, 
so that our construction in Lemma \ref{disjoint5seqlem} 
would lead to a bound $\lvert \DCyl_X (m) \rvert \geq p_X (2m-C)/(2d-1)$. 
We could nevertheless take $c = \log (d!/2)/(2d-1)$, 
which grows in $d$. 
\end{rmrk}

\begin{rmrk}
\normalfont
The same argument also gives a lower bound on the LEF action growth of 
$\lBrack \sigma \rBrack^{\prime}$, 
which is a little stronger than that obtained 
by applying Remark \ref{LEFLAGrthIneqRmrk} 
to the conclusion of Theorem \ref{LBmainthm}. 
Suppose $\pi : B_S(r) \rightarrow \Sym (d)$ is a local embedding. 
Then by Propositions \ref{dirprodLEFprop} and \ref{keyLBprop}, 
$\im(\pi)$ contains a subgroup isomorphic to 
the direct product of $P = p_{\mathbf{x}} (2c r^{1/2}-7)$ copies of $\Alt(5)$, 
which in turn contains the direct product of 
$P$ copies of $C_5$. 
By \cite[Theorem 2]{Johnson}, the minimal degree 
of a faithful permutation representation of the latter is $5 P$. 
Hence $\mathcal{LA}_{\lBrack \sigma \rBrack^{\prime}} (n) 
\succeq p_{\mathbf{x}} (n^{1/2})$. 
\end{rmrk}

\section{Systems of intermediate growth} \label{IntSection}

In this Section we prove Theorem \ref{IntGrowthShiftThm}, 
and deduce Theorems \ref{IntGrowthmainthm} and \ref{UnctbleThm}.
The following example is modelled on the construction in \cite[\S 3]{JLP16}. 
In this Section, ``large'' means larger than a certain absolute constant 
(which we do not compute), 
so as to make true some needed inequalities. 

Choose a real number $r\geq 2$. 
First, fix a large $x$ divisible by $3$. 
We work over the alphabet $A = \lbrace a,b\rbrace$. 
Fix two words $w^{(0)},w'^{(0)} \in A^{\ast}$ of length $x/3$. 
Among all words $w^{(0)}vw'^{(0)}$ for $|v|=x/3$, 
take a subset $\mathcal{C}_{0}$  with $|\mathcal{C}_{0}|=x$ 
and such that no factor of any element of $\mathcal{C}_{0}$ 
is equal to $w^{(0)},w'^{(0)}$ 
except for the prefixes and suffixes themselves. 
This is easily achieved: 
for instance taking $w^{(0)} = a^{x/3}$, 
$w'^{(0)} = b^{x/3}$, we can form $\mathcal{C}_{0}$ 
by choosing $v$ from among those words starting in $b$ and 
ending in $a$ (of which there are $2^{(x/3)-2} \geq x$ for $x$ large). 

Assuming that we have already defined $\mathcal{C}_{j}$, 
we are going to define $\mathcal{C}_{j+1}$; 
for all $i$, we set $N_{i}=|\mathcal{C}_{i}|$ 
and $l_{i}$ the length of any element of $\mathcal{C}_{i}$, 
so that $N_{0}=l_{0}=x$. We prove the following for all $j\geq 0$:
\begin{itemize}
\item[(i)] $N_{j},l_{j}$ are large and increasing in $j$; 
$3|N_{j}$, $3|l_{j}$; 
all words in $\mathcal{C}_{j}$ have the same prefix of length $l_j / 3$, 
and the same suffix of length $l_j / 3$;
\item[(ii)] $N_{j+2}\leq \exp\big(2^{r}(\log N_{j})^{r}\big)$, 
$N_{j+1}\geq  N_{j}^{2}$, 
and for $j$ even $N_{j+1}\geq \exp\big(\frac{1}{2}(\log N_{j})^{r}\big)$;
\item[(iii)] $N_{j}<l_{j+1}\leq N_{j}^{2}$.
\end{itemize}
We prove (i) and (iii) by induction with base case $j=0$, for which all claims are true (for $l_{1}$ see below), and (ii) directly.

Fix an ordering of $\mathcal{C}_{j}$, arbitrarily: the elements of $\mathcal{C}_{j}$ are words $u^{(j)}_{i}$, where the index $i$ follows the ordering. 
Define the word: 
\begin{equation*}
u^{(j+1)}=u^{(j)}_{1}u^{(j)}_{2}u^{(j)}_{3}\ldots u^{(j)}_{N_{j}-1}u^{(j)}_{N_{j}}=w^{(j+1)}u^{(j)}_{\frac{1}{3}N_{j}+1}u^{(j)}_{\frac{1}{3}N_{j}+2}\ldots u^{(j)}_{\frac{2}{3}N_{j}}w'^{(j+1)},
\end{equation*}
where $w^{(j+1)}$ and $w'^{(j+1)}$ collect the first and last third of the $u^{(j)}_{i}$, respectively: this is possible since $3|N_{j}$ by induction, and it implies $3|l_{j+1}$ where $l_{j+1}=|u^{(j+1)}|=l_{j}N_{j}$ is large. 
Take a collection $P_{j} \subset \mathrm{Sym}\left(N_{j}/3\right)$: 
if $j$ is even, choose $|P_{j}|=3\left\lfloor\exp\big((\log N_{j})^{r}\big)/3\right\rfloor$, otherwise choose $|P_{j}|=N_{j}^{2}$ 
(note that this is possible for $x$ large). 
Then, define $\mathcal{C}_{j+1}$ to be the set of all words 
$\lbrace w^{(j+1)}v_{\pi} ^{(j+1)}w'^{(j+1)} : \pi \in P_j \rbrace$ 
where for $\pi \in P_j$, $v_{\pi} ^{(j+1)}$ is obtained by permuting the factors  
$u^{(j)}_{i}$ of $u^{(j+1)}$ 
with $i\in\left(N_{j} / 3,2N_{j} / 3\right]$ 
according to $\pi$; that is: 
\begin{equation*} 
v_{\pi} ^{(j+1)} = u^{(j)}_{\frac{1}{3}N_{j}+\pi(1)}u^{(j)}_{\frac{1}{3}N_{j}+\pi(2)}\ldots u^{(j)}_{\frac{1}{3}N_{j} + \pi (\frac{1}{3}N_{j})}
\end{equation*}
By definition we have $3|N_{j+1}$ and $N_{j+1}$ large; 
$l_{j+1} > l_j$, $N_{j+1} > N_j$, and for all $j\geq 0$,
\begin{align*}
 & N_{j+2}\leq\begin{cases} \exp\big(2(\log N_{j})^{r}\big) & (2|j) \\ 
\exp\big((\log N_{j}^{2})^{r}\big) & (2\nmid j) \end{cases}
 \leq \exp\big(2^{r}(\log N_{j})^{r}\big); \\
 & N_{j+1}\geq\min\big\{N_{j}^{2},\exp\big((\log N_{j})^{r}\big)-3\big\}
 =N_{j}^{2}; \\
2|j \ \Rightarrow \ & N_{j+1}
\geq \exp((\log N_{j})^{r}\big)-3
\geq \exp\big((\log N_{j})^{r} / 2\big); \\
 & l_{1}=x^{2}=N_{0}^{2}>N_{0}; \\
j\geq 1 \ \Rightarrow \ & N_{j}<l_{j+1}=l_{j}N_{j}\leq N_{j-1}^{2}N_{j}\leq N_{j}^{2}, 
\end{align*}
so (i)-(iii) do indeed hold. 
Two key features of this construction are that, for all $j$: 
\begin{itemize}
\item[(a)] All words in $\mathcal{C}_j$ 
have the same prefix $w^{(j)}$ and suffix $w'^{(j)}$ 
of length $\frac{1}{3}l_{j}$; 
\item[(b)] Every word in $\mathcal{C}_{j+1}$ 
is the product of all the words from $\mathcal{C}_j$ (in some order).
\end{itemize}

We now construct a uniformly recurrent non-periodic word 
$\mathbf{x} = \mathbf{x}(r)$, 
such that the subshift $X_r = \overline{\mathcal{O}(\mathbf{x})}$ 
satisfies the conditions of Theorem \ref{IntGrowthShiftThm}. 
Consider the words $x^{(j)} = u_{N_j / 3} ^{(j)}$. 
Then $x^{(j)}$ is the $K_j = (l_j (N_j / 3 - 1) + 1)$th $l_j$-factor of 
$x^{(j+1)}$, for all $j$, and $2 \leq K_j \leq l_{j+1} - l_j$. 
For any $1 \leq M_0 \leq l_0$, apply Lemma \ref{LimitWordLem} 
to the sequence $x^{(j)}$ (with $L_j = l_j$) 
to obtain $\mathbf{x}$. 


\begin{propn} \label{IntCplxLB}
The complexity function $p_{\mathbf{x}}$ of $\mathbf{x}$ satisfies 
$p_{\mathbf{x}}(l_j) \geq \exp\big( (\log l_j)^{r}/2^{r+1}\big)$, 
for all $j \in \mathbb{N}$ odd. 
\end{propn}

\begin{proof}
Since $u_{N_j / 3} ^{(j+1)} \in \mathcal{C}_{j+1}$ 
is a factor of $\mathbf{x}$, 
we have by (b) above that all elements of $\mathcal{C}_j$ 
are distinct factors of $\mathbf{x}$ of length $l_j$, 
so $p_{\mathbf{x}}(l_j)\geq N_{j}$ for all $j \in \mathbb{N}$. 
For $j$ odd, it follows that: 
\begin{align*}
p_{\mathbf{x}}(l_j) & \geq N_{j}
\geq \exp\big((\log N_{j-1})^{r} / 2\big)
\geq \exp\big((\log\sqrt{l_{j}})^{r} / 2\big)
= \exp\big( (\log l_j)^{r}/2^{r+1}\big)
\end{align*}
by (ii) and (iii) above. 
\end{proof}

\begin{propn} \label{IntRecUB}
The recurrence function $R_{\mathbf{x}}$ of $\mathbf{x}$ satisfies: 
\begin{equation*}
R_{\mathbf{x}}(n)\leq C \exp \left( 4^{r+1}(\log n)^{r}\right)
\end{equation*}
for some $C>0$ and all $n \in \mathbb{N}$.
\end{propn}

\begin{proof}
Fix any factor $w$ of $\mathbf{x}$ of length $n$, 
and suppose that $j$ is such that $n< 2 l_{j} / 3$. 
$w$ is a factor of some $x^{(k)} = u_{N_k / 3} ^{(k)}$, 
for $k > j$, and applying (b) above repeatedly, 
$x^{(k)}$ is expressible as a product of words $u_i ^{(j)}$. 
Therefore, there are indices $i_1$ and $i_2$ such that 
$w$ is a factor of some $u^{(j)}_{i_1}u^{(j)}_{i_2}$; 
moreover it intersects the middle third of at most one of 
$u^{(j)}_{i_1}$ or $u^{(j)}_{i_2}$, so by (a) above, 
it sits entirely inside either $u^{(j)}_{i_1}w^{(j)}$ or $w'^{(j)}u^{(j)}_{i_2}$. 

Now let $v$ be any factor $\mathbf{x}$ of length $3 l_{j+1}$. 
As before, $v$ is a factor of some $x^{(k')}=u_{N_{k'} / 3} ^{(k')}$, 
for $k' > j+1$, and applying (b) above repeatedly, 
$x^{(k')}$ is expressible as a product of words $u_i ^{(j+1)}$. 
Therefore, there are indices $i' _1$ and $i' _2$ such that 
$u_{i' _1} ^{(j+1)} u_{i' _2} ^{(j+1)}$ is a factor of $v$. 
By (b) above, and the fact that every $u_i ^{(j)}$ 
has $w^{(j)}$ as a prefix and $w'^{(j)}$ as a suffix, 
it follows that $u^{(j)}_{i_1}w^{(j)}$ and $w'^{(j)}u^{(j)}_{i_2}$ 
are factors of $u_{i' _1} ^{(j+1)} u_{i' _2} ^{(j+1)}$, 
hence $w$ is a factor of $v$. 
That is, $R_{\mathbf{x}}(n)\leq 3 l_{j+1}$. 

Partitioning the integers, we have $R_{\mathbf{x}}(n)\leq 3l_{j+1}$ for any 
$n\in\left[2l_{j-1} /3,2l_{j}/3 \right)$. For $j\geq 2$, we have: 
\begin{align*}
\log R_{\mathbf{x}}(n) - \log 3 & \leq \log l_{j+1} \leq 2 \log N_j \leq 2^{r+1}(\log N_{j-2})^{r} \\
& < 2^{r+1}(\log l_{j-1})^{r}  \leq 2^{r+1}(\log \big(3n/2)\big)^{r}  \leq 4^{r+1} (\log n)^r
\end{align*}
by (ii) and (iii), and our bounds on $n$, 
as required.
\end{proof}

\begin{proof}[Proof of Theorem \ref{IntGrowthShiftThm}]
Let $\mathbf{x} = \mathbf{x}(r)$ be as above. 
Set $X_r = \overline{\mathcal{O}(\mathbf{x})}$. 
By Remark \ref{minseqindeprmrk}, 
items (i) and (ii) follow from Propositions \ref{IntRecUB} and \ref{IntCplxLB}, 
respectively, with $n_i ^{(r)} = l_{2i+1}$. 
\end{proof}

\begin{proof}[Proof of Theorem \ref{IntGrowthmainthm}]
Recall that $N! \leq \exp (N \log N)$ for all $N$. 
Applying Theorem \ref{UBmainthm} to 
$\Gamma^{(r)} = \lBrack \sigma \rBrack^{\prime}$, 
where $(X_r,\sigma)$ is as in Theorem \ref{IntGrowthShiftThm}, we have: 
\begin{equation} \label{IntGrowthUB1}
\mathcal{L}_{\Gamma^{(r)}} ^S (n) 
\leq \exp \Big( 2 R_X(Cn)\log \big( 2R_X(Cn)\big) \Big)
\end{equation}
for all $n$ and for some $C>0$, 
where $X=X_r$, 
so that $R_X (n) \leq \exp \big( C_r (\log n)^r \big)$.
Hence for $n\geq 2$: 
\begin{equation*}
\log \big( 2R_X(Cn) \big) 
\leq C_r (\log n + \log C)^r + \log 2 \leq C_r ' (\log n)^r
\end{equation*}
for a possibly larger constant $C_r '$, 
so that: 
\begin{align*}
\log \big( 2R_X(Cn)\big) + \log \log \big( 2R_X(Cn)\big) 
& \leq C_r ^{\prime} (\log n)^r + r \log \log n + \log C_r ' \\
& \leq C_r '' (\log n)^r
\end{align*}
again, for $C_r ''$ a possibly larger constant. Thus, by (\ref{IntGrowthUB1}), 
$\mathcal{L}_{\Gamma^{(r)}} ^S (n) \leq \exp (\exp (C_r '' (\log n)^r))$, 
and we have (i). For (ii), suppose for a contradiction that 
$2 \leq r' < r$ and that 
$C , C' > 0$ are such that, for all $n$ sufficiently large, 
\begin{equation} \label{IntGrowthLB1}
\mathcal{L}_{\Gamma^{(r)}} ^S (n) 
\leq \exp \Big( \exp \big( C(\log C'n)^{r^{\prime}} \big) \Big)
\end{equation}
By Theorem \ref{LBmainthm} and Theorem \ref{IntGrowthShiftThm} (ii), we have: 
\begin{equation*}
\mathcal{L}_{\Gamma^{(r)}} ^S (n) 
 \geq \exp \left( c p_X \left( c \sqrt{n_i ^{(r)}}\right) \right)
 \geq \exp \left( c \exp \left( c_r \left( \frac{1}{2}\log n_i ^{(r)} + \log c \right)^r \right) \right)
\end{equation*}
for some $c>0$ and all $i \in \mathbb{N}$. Hence by (\ref{IntGrowthLB1}), 
\begin{center}
$C(\log n_i ^{(r)} + \log C' )^{r^{\prime}} 
\geq c_r \left( \frac{1}{2}\log n_i ^{(r)} + \log c \right)^r + \log c$
\end{center}
for all $i$ sufficiently large, a contradiction. 
\end{proof}

\begin{proof}[Proof of Theorem \ref{UnctbleThm}]
Let $\mathcal{F} = \lbrace \Gamma^{(r)} : r \geq 2 \rbrace$, 
with $\Gamma^{(r)}$ as in Theorem \ref{IntGrowthmainthm}. 
Let $2 \leq r^{\prime} < r$, and suppose for a contradiction that 
$\mathcal{L}_{\Gamma^{(r)}} \preceq \mathcal{L}_{\Gamma^{(r^{\prime})}}$. 
By Theorem \ref{IntGrowthmainthm} (i), 
\begin{center}
$\mathcal{L}_{\Gamma^{(r)}} (n) 
\preceq \exp \Big( \exp \big( C_{r^{\prime}}(\log n)^{r^{\prime}} \big) \Big)$, 
\end{center}
contradicting Theorem \ref{IntGrowthmainthm} (ii). 
\end{proof}




\section*{Acknowledgements}

Both authors were supported by ERC grant no.~648329 ``GRANT'', during their permanence at Georg-August-Universit\"at G\"ottingen. 
The second author was supported by the Emily Erskine Endowment Fund and a postdoctoral fellowship from the Einstein Institute of Mathematics, during his permanence at the Hebrew University of Jerusalem.

\footnotesize
\ \\
\textsc{H.~Bradford}. \ \ Christ's College, 
St Andrew's Street, Cambridge, England, CB2 3BU.\\
\texttt{hb470@cam.ac.uk}\\ \\
\textsc{D.~Dona}. \ \ Einstein Institute of Mathematics, Edmond J. Safra Campus Givat Ram, The Hebrew University of Jerusalem, 9190401 Jerusalem, Israel.\\
\texttt{daniele.dona@mail.huji.ac.il}

\end{document}